\documentclass[10pt, reqno]{amsart}

\theoremstyle{theorem}
\newtheorem{theorem}{Theorem}
\newtheorem{corollary}[theorem]{Corollary}

\theoremstyle{definition}

\newtheorem*{remark}{Remark}

\newcommand{\M}{\mathbb{M}}%matrices
\newcommand{\N}{\mathbb{N}}%natural numbers
\newcommand{\D}{\mathbb{D}}%unit disk
\newcommand{\C}{\mathbb{C}}%complex plane

\newcommand{\cB}{\mathcal{B}}%bounded linear operators
\newcommand{\cF}{\mathcal{F}}%Frobenius norm

\DeclareMathOperator{\adj}{adj}
\DeclareMathOperator{\re}{Re}

\DeclareMathOperator{\trace}{trace}%the trace of a matrix

\newcommand{\Mydef}{\stackrel{\mbox{\footnotesize{\rm def}}}{=}}

\begin{document}
\title{Maximum Principles for Matrix-Valued Analytic Functions}
\markright{Maximum Principles}
\author{Alberto A. Condori}
\maketitle
\begin{abstract}
			To what extent is the maximum modulus principle for scalar-valued 
			analytic functions valid for matrix-valued analytic functions?  
			In response, we discuss some maximum \emph{norm} principles 
			for such functions that do not appear to be widely known, deduce 
			maximum and minimum principles for their singular values, and
			make some observations concerning resolvents and matrix exponentials.
\end{abstract}

\section{Introduction.}

The maximum modulus principle (MMP) is a fundamental result in complex analysis.
It is often used to deduce other important results such as the fundamental 
theorem of algebra, the open mapping theorem (i.e., analytic functions map open 
sets to open sets), Schwarz's lemma, the Phragm\'{e}n--Lindel\"{o}ff principle, 
etc.  One of its various formulations states that if $f$ is a \emph{scalar-valued} 
function, analytic on a region $\Omega$ (i.e., a nonempty open connected subset) 
of the complex plane $\C$, whose \emph{modulus} attains a local maximum in $\Omega$,
then $f$ is constant on $\Omega$.  
For a proof of the MMP, we refer the reader to \cite[Chapter 10]{R}. 

Many differential equations encountered in science and engineering lead to the 
consideration of \emph{matrix-valued} functions, that is, functions with range
in the set $\M_{n}$ of $n\times n$ matrices, $n>1$, with entries in $\C$.  For 
instance, the standard model of an RLC circuit in electrical engineering admits 
the formulation $x^{\prime}(t)=A\cdot x(t)$, where $A \in \M_{n}$ and $x$ is a 
function with values in $\C^{n}$.  The vector-valued solutions $x(t)=\exp(tA)x_0$ 
to such an equation (with $x_0\in\C^{n}$) depend on the matrix-valued function 
$t\mapsto\exp(tA)=\sum_{k=0}^{\infty}A^{k}t^{k}/k!$, and the decay of these 
solutions is controlled by its operator norm $\|\exp(tA)\|$.  As usual, 
$\|T\|=\sup\{\|Tv\|_{\C^{n}}: \|v\|_{\C^{n}}=1\}$ is the \emph{operator norm}
of $T\in\M_n$ induced by the Euclidean norm on $\C^{n}$, namely $\|v\|_{\C^{n}}
=\left(|v_1|^2+\cdots+|v_n|^2\right)^{1/2}$ when $v=(v_1,\ldots,v_n)$.

In linear algebra, too, matrix-valued functions arise (implicitly) in the study of 
eigenvalues, i.e., the spectrum $\sigma(A)$ of $A\in\M_n$. After all, $\lambda\in\C$ 
satisfies $Av=\lambda v$ for some nonzero vector $v\in\C^{n}$ if and only if the 
\emph{resolvent} function $z\mapsto (A-zI)^{-1}$ has a singularity at $z=\lambda$, 
i.e., $A-\lambda I$ is not invertible.  (Throughout, $I=I_{n}$ denotes the identity in 
$\M_{n}$.)  Since the spectrum is often insufficient for the analysis of non-normal 
matrices (see \cite{T}), focus has shifted to the study\footnote{Equivalently, one 
may study the so-called ``pseudospectra'' of $A$.  For an overview of that subject, 
see \cite{TE}.} of the norm of the resolvent $\|(A-zI)^{-1}\|$.  For instance, the 
norm of the resolvent alone characterizes when $A$ is a normal matrix \cite{BC}.

Thus, it is of interest to study the (operator) norms of the matrix-valued 
functions $\exp(zA)$ and $(A-zI)^{-1}$.  As can be expected, these functions 
are analytic\footnote{Recall that a function $F:\Omega\to\M_{n}$ is analytic if, 
for each $z_0\in\Omega$, there is a member of $\M_n$, denoted by $F^{\prime}(z_0)$, 
such that $\|(z-z_0)^{-1}\{F(z)-F(z_0)\}-F^{\prime}(z_0)\|\to 0$ as $z\to z_0$. 
It can be shown that $F:\Omega\to\M_{n}$ is analytic if and only if $F$ is 
``entry-wise analytic,'' i.e., every entry of $F(z)$ is an analytic function 
on $\Omega$.} in regions of $\C$, the entire plane $\C$, and $\C\backslash\sigma(A)$, 
respectively.  The fact these functions are analytic leads one to question the extent 
to which the MMP for scalar-valued functions is valid for $\|F(z)\|$, where $F$ is 
any matrix-valued analytic function.  The purpose of this article is to find sufficient
conditions, say involving the norm of a matrix-valued analytic function, that guarantee
that the function is constant.
 
In Section \ref{sectionMaxNormPrinciples}, we state and discuss some maximum norm 
principles for matrix-valued analytic functions.  Although it has been long known that 
a direct analog of the MMP fails in the context of matrix-valued functions in which
the operator norm plays the role of the modulus, we find a suitable analog.  Stated 
roughly, if $F:\Omega\to\M_n$ is such that $\|F(z)\|$ attains a maximum at some 
$z_0\in\Omega$, then there is a direction in which $F(z)$ is constant (although
$F(z)$ need not be) namely that of any maximizing vector of $F(z_0)$ (see Theorem 
\ref{MONPmaxVector}).  We rediscovered this result originally noted by Brown and 
Douglas in \cite{BD} and use it to describe the structure of the function 
$F(z)$ (see Theorem \ref{MONPmaxVectorFactorization}).  Since the result lends itself 
to iteration, we make natural assumptions on the function's singular values and explore 
the consequences further in Section \ref{sectionMaxSkPrinciples}.  One of the section's
main results (see Corollary \ref{constantFThm}) illuminates the equivalence of two
apparently distinct statements to the single statement that the matrix function $F(z)$ 
is constant:  the Frobenius norm of $F(z)$ attains a maximum, and every singular
value of $F(z)$ attains a maximum (at possibly distinct points).

Once the maximum singular-values principle is established in Section 
\ref{sectionMaxSkPrinciples}, we proceed to prove a minimum
singular-values principle in Section \ref{sectionMinSkPrinciples}.  
That result (Theorem \ref{MSVP}) is, in a sense, an analog of the 
well-known minimum modulus principle of complex analysis in the context 
of matrix-valued functions.  Finally, in Section \ref{sectionExamples},
we discuss the implications of our results in the context of the resolvent 
and the matrix exponential which involve their largest and smallest singular 
values.

It is worth mentioning that analytic matrix-valued functions appear in many other
areas such as the harmonic analysis of operators on a Hilbert space (e.g., finite-rank 
perturbations of self-adjoint and unitary operators), and consequently in mathematical 
physics (e.g., Schr\"{o}dinger operators); roughly, problems concerning spectral properties
of an operator are often solved through the consideration of an analytic matrix-valued 
function defined on the upper-half plane, i.e. the so-called ``characteristic function.'' 
Due to the scope of the paper, the reader is referred to the survey \cite{N} and all 
references therein for further details.

We also remark that the results of this article could be written in the more general 
framework of operator-valued functions $F:\Omega\to\cB(H)$, where $H$ is a complex 
Hilbert space, or that of vector-valued functions $F:\Omega\to B$, where $B$ is a 
complex Banach space.  However, all statements in this article are kept in the context 
of matrix-valued functions so that the results are easier to read and appeal to a wider 
audience.

\section{Maximum norm principles.}\label{sectionMaxNormPrinciples}

To find a suitable analog of the MMP for matrix-valued functions, it is reasonable 
to first test whether known proofs of the MMP can be easily adapted when replacing 
modulus with operator norm. One such proof of the MMP appears in \cite[Chapter 10]{R}.
In it, the identity $|w|^2=w\bar{w}$ ($w\in\C$) appears, and although the operator
norm of $T\in\M_n$ does not readily provide a direct analog for $\|T\|^{2}$, the
Frobenius norm does.  In fact,\footnote{As usual, $T^{*}$ 
denotes the conjugate transpose of the matrix $T$.}
\begin{equation}\label{FrobeniusNormDef}
			\|T\|_{\cF}^2=\trace(T^{*}T)\;\text{ for }T\in\M_n,
\end{equation}
when $\|T\|_{\cF}$ is the Frobenius (Hilbert--Schmidt) norm of $T$, and an argument
analogous to the proof of the MMP in \cite{R} (that relies on \eqref{FrobeniusNormDef}) 
gives the following result.

\begin{theorem}[Maximum Frobenius Norm Principle]\label{MFNP}
			Let $\Omega$ be a region of $\C$ and let 
			$F:\Omega\to\M_n$ be analytic.  If $\|F(z)\|_{\cF}$ assumes its maximum 
			at some $z_0\in\Omega$, then $F(z)=F(z_0)$ for all $z\in\Omega$.
\end{theorem} 
Despite its provision of a direct analog of the MMP for matrix-valued functions, 
in applications, it is the operator norm that is of interest, not the Frobenius norm.  
Unfortunately, the conclusion of Theorem \ref{MFNP} need not hold when the Frobenius 
norm is replaced by another matrix norm.  For example, let $\D$ denote the open unit 
disk centered at the origin, let $g:\D\to\D$ be analytic (e.g., $g(z)=z$), and consider 
the $2\times 2$ matrix-valued function
\begin{equation}\label{firstEx}
			F(z)=\left[
			\begin{array}{cc}
						1	&	0\\
						0	&	g(z)
			\end{array}\right].
\end{equation}
Notice that the operator norm of $F(z)$ satisfies 
\[	\|F(z)\|^2=\max\{1,|g(z)|^2\}=1\;\text{ for all }z\in\D,	\]
even though $F(z)$ is not a constant function.  Nevertheless, one can prove a 
weakened version for \emph{any} norm.

\begin{theorem}[Maximum Norm Principle]\label{MNP}
			Let $\Omega$ be a region of $\C$ and let 
			$F:\Omega\to\M_n$ be analytic.  If $\|F(z)\|$ attains its maximum 
			in $\Omega$, then $\|F(z)\|$ is constant on $\Omega$.
\end{theorem}

Theorem \ref{MNP} is well known and a proof can be found in \cite[Section III.14]{DS};
we provide a different short proof based on a well-known consequence 
of the Hahn--Banach theorem on linear functionals, namely if $X$ is any 
normed space and $x\in X$ is nonzero, then there is a bounded linear functional 
$\Lambda$ on $X$ such that $\|\Lambda\|=1$ and $\Lambda(x)=\|x\|$.  For further 
details and a simple proof of this fact, see \cite[Chapter 5]{R}.

\begin{proof}[Proof of Theorem \ref{MNP}]
			Assume there is a $z_0\in\Omega$ such that $\|F(z)\|\leq\|F(z_0)\|$ for all 
			$z\in\Omega$ and, without loss of generality, that $\|F(z_0)\|\neq0$.  Then
			we can choose a bounded linear functional $\Lambda:\M_n\to\C$  of norm $1$ so 
			that $\|F(z_0)\|=\Lambda(F(z_0))$.  By continuity of $\Lambda$ and analyticity
			of $F$, $\Lambda(F(z))$ defines an analytic function on $\Omega$, and
			\[	|\Lambda(F(z))|\leq\|F(z)\|\leq\|F(z_0)\|=|\Lambda(F(z_0))|.	\]
			It follows now from the usual MMP that $\Lambda(F(z))$ must be constant
			throughout $\Omega$ and
			\[	\|F(z)\|\geq\Lambda(F(z))=\Lambda(F(z_0))=\|F(z_0)\|\geq\|F(z)\|
					\quad\text{ for all }z\in\Omega.	\]
			Thus, $\|F(z)\|=\|F(z_0)\|$ for all $z\in\Omega$.
\end{proof}

The conclusion of the maximum norm principle above may be seen as 
unsatisfactory because it gives limited information about the structure 
of $F(z)$ itself.  This is not at all surprising; after all, the theorem 
holds for \emph{any} norm.  So, from now on we use the \emph{operator norm} 
exclusively in an effort to gain more information about the function $F$.

A useful property of the operator norm of a matrix is that given any $A\in\M_n$, 
there is a unit vector $x_0\in\C^n$, called a \textbf{maximizing vector} for 
$A$, so that $\|Ax_0\|=\|A\|$; in other words, matrices attain their operator norm
at some vector in the unit ball of $\C^{n}$.  This is a consequence of the compactness 
of the closed unit ball of $\C^{n}$.

Recently, we rediscovered a maximum operator norm principle due to Brown 
and Douglas.  In \cite[Theorem 4]{BD}, the authors proved that if $F(z)$ 
is a nonconstant matrix-valued analytic function whose operator norm attains its
maximum, then there is a direction $x_0$ in which $F(z)x_0$ is constant. Our version 
reads as follows.

\begin{theorem}[Maximum Operator Norm Principle, cf. \cite{BD}]\label{MONPmaxVector}
			Let $\Omega$ be a region of $\C$ and let 
			$F:\Omega\to\M_n$ be analytic.  If there is a $z_0\in\Omega$ so that 
			$\|F(z)\|\leq \|F(z_0)\|$ for all $z\in\Omega$ and $x_0$ is a maximizing
			vector for $F(z_0)$, then $F^{(k)}(z_0)x_0=0$ for every $k\geq1$.  
			In particular, $F(z)x_0$ is constant on $\Omega$.
\end{theorem}

The conclusion\footnote{It is worth mentioning that our version of Theorem 
\ref{MONPmaxVector} also complements a result due to Daniluk in \cite{D}.} 
of Theorem \ref{MONPmaxVector} here is, at first sight, a slight improvement to 
that in Theorem 4 (part (1)) of \cite{BD}; after all, using a series expansion 
of $F(z)$, the condition $F^{(k)}(z_0)x_0=0$ for $k\geq 1$ easily implies that 
$F(z)x_0$ is constant on $\Omega$.  In fact, the reverse implication is also 
true and a justification can be made using series, too.  On the other hand,
although our series proof of Theorem \ref{MONPmaxVector} below is not as
short as that of Brown and Douglas, it elucidates the consideration of
maximizing vectors $x_0$ (see \eqref{inequalityForCks} below).

\begin{proof}[Proof of Theorem \ref{MONPmaxVector}]
			Let $R>0$ be such that $D(z_0;R)\subseteq\Omega$.  Then $F(z)$ admits a 
			power series representation on $D(z_0;R)$, say
			\begin{equation}\label{powerSeries}
						F(z)=\sum_{k=0}^{\infty}C_k(z-z_0)^{k}, 
			\end{equation}
			where $C_k\in\M_n$ for $k\geq0$.  For any vector $x$, 
			\[	\|F(z)x\|^2=\sum_{j,k\geq0}
					(z-z_0)^{j}(\overline{z-z_0})^{k}\langle C_j x,C_k x\rangle	\]
			by continuity of the inner product and so
			\begin{equation}\label{meanvalueIdentity}
						\frac{1}{2\pi}\int_{0}^{2\pi}\|F(z_0+re^{it})x\|^2\,dt
						=\sum_{k=0}^{\infty}\|C_k x\|^{2}r^{2k}
			\end{equation}
			for any $r\in(0,R)$.
			
			Now, since $\|F(z)\|\leq \|F(z_0)\|=\|C_0\|$ for all $z\in\Omega$, 
			it follows from \eqref{meanvalueIdentity} that
			\begin{equation}\label{inequalityForCks}
						\sum_{k=0}^{\infty}\|C_k x\|^{2}r^{2k}
						=\frac{1}{2\pi}\int_{0}^{2\pi}\|F(z_0+re^{it})x\|^2\,dt
						\leq\|C_0\|^2\|x\|^2
			\end{equation}
			for any vector $x$ and $r\in(0,R)$.  Let $x_0$ be a maximizing vector 
			for $C_0$.  Replace $x$ by $x_0$ in \eqref{inequalityForCks}, and conclude
			\[	\|C_0\|^2+\sum_{k=1}^{\infty}\|C_k x_0\|^{2}r^{2k}
					\leq \|C_0\|^2\|x_0\|^2=\|C_0\|^2,	\]
			and $F^{(k)}(z_0)x_0=C_k x_0=0$ for every $k\geq 1$.
			In particular, by \eqref{powerSeries}, $F(z)x_0=C_0x_0=F(z_0)x_0$ for all 
			$z\in D(z_0;R)$ and so, by the identity theorem (e.g., \cite[Theorem 10.18]{R}), 
			$F(z)x_0=F(z_0)x_0$ for all $z\in\Omega$.
\end{proof}

\begin{remark}
			Note that \emph{the conclusion of Theorem \ref{MONPmaxVector} alone} implies 
			that $\|F(z)\|$ has a \emph{minimum} at $z_0$; after all, if $z\mapsto F(z)x_0$ 
			is constant on $\Omega$ for some maximizing vector $x_0$ of $F(z_0)$, then
			\[	\|F(z_0)\|=\|F(z_0)x_0\|=\|F(z)x_0\|\leq\|F(z)\|\;
					\text{ for all }z\in\Omega. 	\]
			Hence, the conclusion of Theorem \ref{MONPmaxVector} is stronger than that 
			of the maximum norm principle (when using the operator norm) because it implies 
			that any maximizing vector $x_0$ for $F(z_0)$ is also a maximizing vector for 
			$F(z)$, and $F(z)$ has constant norm equal to that of $F(z_0)$ for all $z\in\Omega$.
\end{remark}

The observation made in the remark leads one to the following factorization. 

\begin{theorem}\label{MONPmaxVectorFactorization}
			Let $\Omega$ be a region of $\C$ and let $F:\Omega\to\M_n$ be analytic.  
			If there is a $z_0\in\Omega$ so that $\|F(z)\|\leq \|F(z_0)\|$ for all 
			$z\in\Omega$, then there are $n\times n$  (constant) unitary\footnote{Recall 
			that $A\in\M_{n}$ is said to be \textbf{unitary} if $A^{*}A=AA^{*}=I$.} matrices 
			$U$ and $V$, and an analytic function $G:\Omega\to\M_{n-1}$, such that
			\begin{equation}\label{svdFzWeak}
						F(z)=U\left[
						\begin{array}{cc}
									\|F(z_0)\|	&	0\\
									0						&	G(z)
						\end{array}\right]V.
			\end{equation}
\end{theorem}

Roughly, in the case of $2\times 2$ matrices, Theorem \ref{MONPmaxVectorFactorization} 
states that when $F(z)$ is nonconstant, analytic, and achieves its maximum operator norm, 
say equal to $1$, at a point of a region, then there is a nonconstant analytic function 
$g:\Omega\to\D$ such that
\[	F(z)=\left[
		\begin{array}{cc}
					1	&	0\\
					0	&	g(z)
		\end{array}\right]	\]
up to multiplication by (constant) unitary matrices on the right and the left.
Hence, in a sense, the example given in \eqref{firstEx} is essentially the only
example of a nonconstant $2\times 2$ matrix function whose operator norm
achieves a maximum value of $1$.

\begin{proof}[Proof of Theorem \ref{MONPmaxVectorFactorization}]
			Without loss of generality, we assume $\|F(z_0)\|=1$.  By Theorem 
			\ref{MONPmaxVector}, if $x_0$ is a maximizing vector for $F(z_0)$, 
			then the vector function $z\mapsto F(z)x_0$ is constant on $\Omega$.  
			Recalling that $\|v\|_{\C^{n}}^2=v^{*}v$ for any $v\in\C^{n}$ and
			choosing $y_0=F(z_0)x_0$, we obtain
			%In view of \eqref{eigenvectorForModulus},
			\[	\|y_0\|^2=\|x_0\|^2=1\;\text{ and }\;
					y_0^{*}F(z)x_0=y_0^{*}F(z_0)x_0=1\;\text{ for all }z\in\Omega.	\]
			Let $X_0$ and $Y_0$ be (constant) $n\times n$ unitary matrices whose 
			first columns are $x_0$ and $y_0$, respectively.  Then, in matrix blocks,
			\[	Y_{0}^{*}F(z)X_{0}=\left[
					\begin{array}{cc}
								a_{1,1}(z)	&	a_{1,2}(z)\\
								a_{2,1}(z)	&	a_{2,2}(z)
					\end{array}\right],	\]
			where $a_{1,1}(z)=y_0^{*}F(z)x_0=1$.  Furthermore, as $X_0$ and $Y_0$ 
			are unitary, $\|Y_{0}^{*}F(z)X_{0}\|=\|F(z)\|=1$ (or, alternatively, 
			this follows by the remark following the proof of Theorem \ref{MONPmaxVector}).  
			This implies that
			\[	a_{1,2}(z)=0\;\text{ and }\;a_{2,1}(z)=0\;\text{ for all }z\in\Omega	\]
			because the operator norm of an $n\times n$ matrix is an upper bound on 
			the Euclidean (vector) norm of its columns and rows.  In other words, 
			the assumptions on $F(z)$ imply the existence of $n\times n$ constant 
			unitary matrices $X_0$ and $Y_0$ so that
			\[	F(z)=Y_{0}\left[
					\begin{array}{cc}
								1	&	0\\
								0	&	a_{2,2}(z)
					\end{array}\right]X_0^{*}	\]
			where $a_{2,2}(z)$ is an analytic $(n-1)\times(n-1)$ matrix-valued 
			function.  Thus, the desired conclusion follows with $U=Y_0$, 
			$V=X_0^{*}$, and $G(z)=a_{2,2}(z)$.
\end{proof}

\section{Maximum singular value principles.}\label{sectionMaxSkPrinciples}

An attractive feature of Theorem \ref{MONPmaxVectorFactorization} is that it
lends itself to iteration.  Indeed, the lower right block $G(z)$ in \eqref{svdFzWeak}
may very well satisfy the assumptions of Theorem \ref{MONPmaxVectorFactorization}
just as $F(z)$ did.  In this section, we explore this situation and its consequences,
but first review some basic terminology and results concerning singular values.

We begin with the observation that the maximizing vectors for a matrix $A$ admit the 
characterization that $x_0$ is a maximizing vector for $A\in\M_n$ if and 
only if $x_0$ has norm $1$ and $A^{*}Ax_0=\|A\|^2x_0$.  More generally, for a 
vector $x$ (whether it has norm $1$ or not),
\begin{equation}\label{eigenvectorForModulus}
			\|Ax\|=\|A\|\cdot\|x\|\;\text{ if and only if }\;
			A^{*}Ax=\|A\|^2 x.
\end{equation}
A proof of \eqref{eigenvectorForModulus} can be based on the fact that every 
positive semi-definite matrix has a unique positive semi-definite square root 
(e.g., see \cite[Theorem 7.2.6]{HJ}).  To that end, first note that the inequality 
$\|Av\|\leq\|A\|\cdot\|v\|$ valid for all vectors $v$ is equivalent to stating 
that the matrix $\|A\|^2 I-A^{*}A$ is positive semi-definite.  So, 
$\|Ax\|=\|A\|\cdot\|x\|$ holds if and only if $\|(\|A\|^2 I-A^{*}A)^{1/2}x\|=0$, 
or equivalently, $(\|A\|^2 I-A^{*}A)x=0$.  Hence, $x_0$ is a maximizing vector 
of $A$ if and only if it is an eigenvector of $A^*A$ of norm $1$, i.e., 
\eqref{eigenvectorForModulus} holds.

The role played in Theorem \ref{MONPmaxVector} by maximizing vectors for a 
matrix and their alternative characterization as eigenvectors lead directly 
to the consideration of singular values.

Recall that the singular values $s_{k}(A)$, $1\leq k\leq n$, of an $n\times n$ 
matrix $A$ are the nonnegative square roots of the eigenvalues of $A^{*}A$ 
ordered in the nonincreasing order, that is, 
\[	s_{1}(A)\geq s_{2}(A)\geq\cdots\geq s_{n}(A).	\]
In particular, $s_{1}(A)=\|A\|$ (see \eqref{eigenvectorForModulus}) and 
$s_{1}^{2}(A)+s_{2}^2(A)+\cdots+s_{n}^2(A)=\|A\|_{\cF}^{2}$.  The latter 
can be deduced using any singular value decomposition (SVD) of $A$ (e.g.,
\cite[Theorem 7.35]{HJ}) and \eqref{FrobeniusNormDef}.

The following result is a simple consequence of Theorem \ref{MONPmaxVectorFactorization}.

\begin{theorem}\label{singValueMaxPrincipleThm}
			Let $\Omega$ be a region of $\C$ and let 
			$F:\Omega\to\M_n$ be analytic.  Suppose that, for each $k=1,\ldots, n$, 
			the function $z\mapsto s_{k}(F(z))$ attains its maximum value on $\Omega$.
			Then $F(z)$ is constant on $\Omega$.
\end{theorem}

In Theorem \ref{singValueMaxPrincipleThm}, the assumption does \emph{not} require 
that the functions $s_{1}(F(z))$,$\,\ldots\,$, $s_{n}(F(z))$ attain their maximum
values at the same point\footnote{In fact, if the functions 
$s_{1}(F(z))$,$\,\ldots\,$,$s_{n}(F(z))$ attain their maximum values at the same point 
$z_0\in\Omega$, it follows already from Theorem \ref{MFNP} that $F(z)$ must be
constant on $\Omega$.} of $\Omega$; they may assume their respective maxima at 
distinct points $z_1$,$\,\ldots\,$,$z_n\in\Omega$.

\begin{proof}[Proof of Theorem \ref{singValueMaxPrincipleThm}]
			Our proof is by induction on $n$.  When $n=1$, the desired conclusion 
			holds by the MMP.  So, suppose that the result holds for $n=1,\ldots,m-1$ 
			with $m\in\N$.  We now show that it also holds for $n=m$.  
			
			Suppose $F:\Omega\to\M_m$ is analytic, and the function $z\mapsto 
			s_{k}(F(z))$ attains its maximum value on $\Omega$ for each $k=1,\ldots, m$.  
			Let $z_1\in\Omega$ be such that $s_{1}(F(z))\leq s_{1}(F(z_1))$ for all 
			$z\in\Omega$. By Theorem \ref{MONPmaxVectorFactorization}, there are $m\times m$  
			(constant) unitary matrices $U_1$ and $V_1$, and an analytic function 
			$F_1:\Omega\to\M_{n-1}$, such that
			\begin{equation}\label{svdFzWeak}
						F(z)=U_1\left[
						\begin{array}{cc}
									s_{1}(F(z_1))	&	0\\
									0							&	F_1(z)
						\end{array}\right]V_1.
			\end{equation}
			In particular, $s_{k}(F_{1}(z))=s_{k+1}(F(z))$ attains its maximum 
			value on $\Omega$ for each $k=1,\ldots,m-1$.  By the inductive 
			hypothesis, $F_{1}(z)$ must be constant on $\Omega$ and, consequently, 
			$F(z)$ is also constant.
\end{proof}

At first sight, the assumption in Theorem \ref{singValueMaxPrincipleThm} that every 
function $z\mapsto s_{k}(F(z))$ attains its maximum value on $\Omega$ for $k=1,\ldots,n$ 
appears to be different from saying that $\|F(z)\|_{\cF}$ attains its maximum value on 
$\Omega$ in the maximum Frobenius norm principle above.  Based upon the results above 
one may conclude that they are in fact equivalent!

\begin{corollary}\label{constantFThm}
			Let $\Omega$ be a region of $\C$.  The following
			statements are equivalent for an analytic function $F:\Omega\to\M_n$.
			\begin{enumerate}
						\item\label{Fconstant}			
									$F(z)$ is constant on $\Omega$.
						\item\label{constantSk}
									For every $k=1,\ldots, n$, $s_{k}(F(z))$ is constant on $\Omega$.
						\item\label{maxedSk}	
									For every $k=1,\ldots, n$, $s_{k}(F(z))$ attains its maximum value 
									at some $z_k\in\Omega$.
						\item\label{normFconstant}
									$\|F(z)\|_{\cF}$ is constant on $\Omega$.
						\item\label{normFmaxAttained}
									$\|F(z)\|_{\cF}$ attains its maximum value at some $z_0\in\Omega$.
			\end{enumerate}
\end{corollary}
\begin{proof}
			It is evident that $(\ref{Fconstant})\implies(\ref{constantSk})$,
			$(\ref{constantSk})\implies(\ref{maxedSk})$, 
			$(\ref{Fconstant})\implies(\ref{normFconstant})$, and
			$(\ref{normFconstant})\implies(\ref{normFmaxAttained})$.			
			The only nontrivial implications 
			$(\ref{normFmaxAttained})\implies(\ref{Fconstant})$ and 
			$(\ref{maxedSk})\implies(\ref{Fconstant})$
			are consequences of the maximum Frobenius norm principle 
			and Theorem \ref{singValueMaxPrincipleThm}, respectively.
\end{proof}

In view of Corollary \ref{constantFThm} (or Theorem \ref{singValueMaxPrincipleThm}),
if $\Omega$ is region of $\C$ and $F:\Omega\to\M_n$ is a \emph{nonconstant} analytic 
function such that $s_1(F(z))$ attains its maximum on $\Omega$, then there is a largest 
integer $r<n$ such that the functions $s_{1}(F(z))$, $\ldots\,$, $s_{r}(F(z))$ attain their 
maximum values on $\Omega$.  In this case, up to multiplication by (constant) unitary 
matrices on the right and the left, $F(z)$ has the block form
\[	\left[
		\begin{array}{cccccc}
					s_{1}(F(z_1))	&	\ldots	&	0								&	0\\
					\vdots				&	\ddots	&	\vdots					&	\vdots\\
					0							& \ldots	&	s_{r}(F(z_{r}))	&	0\\
					0							&	\dots		&	0								&	F_{r}(z)
		\end{array}\right]	\]
for some (necessarily nonconstant) analytic function $F_{r}:\Omega\to\M_{n-r}$.

A closer look at the proof of Theorem \ref{singValueMaxPrincipleThm} also reveals 
the following refinement of the maximum norm principle. We omit the details.

\begin{corollary}\label{MaxNormRefinement}
			Let $1\leq m\leq n$, let $\Omega$ be a region of $\C$,
			and let $F:\Omega\to\M_n$ be analytic.  Suppose that, for each $k=1,\ldots, m$, 
			the function $s_{k}(F(z))$ attains its maximum value on $\Omega$.  Then 
			$s_{k}(F(z))$ is constant on $\Omega$ for each $k=1,\ldots,m$.
\end{corollary}

Note that for an arbitrary $F(z)$, it may happen that $s_{n}(F(z))$ is constant
while $s_{k}(F(z))$ is not when $1\leq k<n$.  For example, the function
$F:\D\setminus\{0\}\to\M_2$ defined by
\[	F(z)=\left[
		\begin{array}{cc}
					1	&	0\\
					0	&	z^{-1}
		\end{array}\right],	\]
has $s_1(F(z))=|z|^{-1}$ and $s_{2}(F(z))=1$ for all 
$z\in\D\setminus\{0\}$.

As seen in its proof, the key to obtaining the conclusion of Theorem 
\ref{MONPmaxVectorFactorization} relies on choosing a maximizing vector
for $F(z_0)$.  The following theorem is a refinement of Theorem 
\ref{MONPmaxVectorFactorization} that relies on choosing instead 
``all maximizing vectors'' for $F(z_0)$.

\begin{theorem}\label{factoringFzThm}
			Let $\Omega$ be a region of $\C$ and let 
			$F:\Omega\to\M_n$ be analytic.  Suppose there is a $z_0\in\Omega$ 
			so that $\|F(z)\|\leq \|F(z_0)\|$ for all $z\in\Omega$ and 
			set\footnote{Equivalently, $d$ is the dimension of the subspace 
			spanned by the ``right-singular vectors'' associated with the largest 
			singular value of $F(z_0)$.} 
			\begin{equation}\label{rDimEq}
						d=\dim\{x\in\C^{n}: \|F(z_0)x\|=\|F(z_0)\|\cdot\|x\|\}.
			\end{equation}
			Then there are $n\times n$ unitary matrices $U$ and $V$ such that
			\[	F(z)=\|F(z_0)\|U\cdot V\;\text{ when }d=n,	\]
			or, for some analytic function $R:\Omega\to\M_{n-d}$,
			\begin{equation}\label{svdFz}
						F(z)=U\left[
						\begin{array}{cc}
									\|F(z_0)\|\cdot I_{d}	&	0\\
									0											&	R(z)
						\end{array}\right]V\;\text{ when }d<n.
			\end{equation}
			In particular, $z\mapsto F(z)x$ is constant and $F^{(k)}(z_0)x=0$ 
			for all $k\geq 1$ when $x$ satisfies $\|F(z_0)x\|=\|F(z_0)\|\cdot\|x\|$.
\end{theorem}

Note that one could apply Theorem \ref{factoringFzThm} again to the lower-right 
matrix-block function $R(z)$ appearing in \eqref{svdFz}.  More definitively, if 
$s_{1}(F(z))$ attains its maximum at $z_1\in\Omega$, then $d_1$ is the largest 
integer such that 
\[	s_{1}(F(z_1))=s_{d_1}(F(z_1)),	\] 
$s_{d_1+1}(F(z))$ attains its maximum at $z_2\in\Omega$, and $d_2$ is the largest 
integer such that 
\[	s_{d_1+1}(F(z_2))=s_{d_2}(F(z_2)),	\]
then up to multiplication by (constant) unitary matrices on the right and the left, 
$F(z)$ has the block form
\[	\left[
		\begin{array}{ccccc}
					s_{d_1}(F(z_1))\cdot I_{d_1}	&	0	&	0\\
					0	&	s_{d_2}(F(z_2))\cdot I_{d_2}	&	0\\
					0	& 0	&		*					
		\end{array}\right].	\]
Hence, if every function $z\mapsto s_{k}(F(z))$ attains its maximum at some point
of $\Omega$ then, up to multiplication by (constant) unitary matrices on the right 
and the left, $F(z)$ admits the block form
\[	\left[
		\begin{array}{ccccc}
					s_{d_1}(F(z_1))\cdot I_{d_1}	&	0	&	\ldots	&	0\\
					0	&	s_{d_2}(F(z_2))\cdot I_{d_2}	&	\ldots	&	0\\
					\vdots	&\vdots	& \ddots	&	\vdots\\
					0	& 0	&	\ldots	&	s_{d_\kappa}(F(z_\kappa))\cdot I_{d_\kappa}					
		\end{array}\right],	\]
and is hence a constant matrix, as expected by Theorem \ref{singValueMaxPrincipleThm}.

Likewise, a completely analogous argument reveals that the refinement of the maximum 
norm principle in Corollary \ref{MaxNormRefinement} is also a consequence of Theorem 
\ref{factoringFzThm}, because $s_{j}(F(z))=\|F(z_0)\|$ for $j=1,\ldots,d$ and 
$s_{\ell+d}(F(z))=s_{\ell}(R(z))$ for $\ell=1,\ldots, (n-d)$.  We leave the 
details to the reader.

\begin{proof}[Proof of Theorem \ref{factoringFzThm}]
			Let $D_{z_0}$ be the diagonal matrix whose main diagonal entries are the 
			singular values of $F(z_0)$ listed in nonincreasing order.  Then we 
			may let $U_{z_0}$ and $V_{z_0}$ be unitary matrices such that 
			$F(z_0)=U_{z_0}D_{z_0}V_{z_0}$ (i.e., an SVD for $F(z_0)$).  Let $r$ denote 
			the largest positive integer such that $s_{r}(F(z_0))=\|F(z_0)\|$.  Note 
			that, by \eqref{eigenvectorForModulus}, a vector $x$ satisfies
			$\|F(z_0)x\|=\|F(z_0)\|\cdot\|x\|$ if and only if 
			$V_{z_0}^{*}(\|F(z_0)\|^{2}I-D_{z_0}^{2})V_{z_0}x=0$, or equivalently,			
			$x$ belongs to the linear span of first $r$ columns of $V^{*}_{z_0}$
			because $V^{*}_{z_0}$ is unitary.  Thus, $r=d$ with $d$ as in \eqref{rDimEq}.
						
			Now, consider the function $G(z)=U_{z_0}^{*}F(z)V_{z_0}^{*}$.
			Clearly, $G$ is analytic on $\Omega$ and satisfies
			\[	\|G(z)\|\leq\|F(z_0)\|\;\text{ for all }z\in\Omega.	\]			
			Since the $\C^{n}$ norm of every column (and row) of a matrix is bounded 
			by its operator norm, the modulus of every (analytic) entry $G_{i,j}(z)$ 
			is also bounded by $\|F(z_0)\|$.  Moreover, if $1\leq i\leq r$, then
			$G_{i,i}(z_0)=\|F(z_0)\|$ and so $G_{i,i}(z)=\|F(z_0)\|$ for all 
			$z\in\Omega$ by the (usual) MMP.  In particular, the first $r$ columns 
			and $r$ rows of $G(z)$ have $\C^{n}$ norm at least $\|F(z_0)\|$.
			Therefore, $G_{i,j}(z)=0$ when $i\neq j$ and $1\leq i,j\leq r$.  In other
			words, using matrix blocks, this shows that
			\[	F(z)=U_{z_0}G(z)V_{z_0}=U_{z_0}\left[
					\begin{array}{cc}
								\|F(z_0)\|\cdot I_{r}	&	0\\
								0											&	R(z)
					\end{array}\right]V_{z_0},	\]
			for some analytic function $R:\Omega\to\M_{n-r}$ when $r<n$, while
			$F(z)=\|F(z_0)\|U_{z_0}V_{z_0}$ when $r=n$.  This completes the proof 
			of \eqref{svdFz}.
			
			Finally, if $e_1,\ldots,e_n$ denotes the standard basis for $\C^{n}$
			and $k\geq 1$, then the $j$th column $V^{*}_{z_0}e_{j}$ of $V^{*}_{z_0}$ 
			satisfies $F(z)V^{*}_{z_0}e_{j}=\|F(z_0)\| U_{z_0}e_{j}$ and
			\[	F^{(k)}(z)V^{*}_{z_0}e_{j}=U_{z_0}\left[
					\begin{array}{cc}
								0\cdot I_{r}	&	0\\
								0							&	R^{(k)}(z)
					\end{array}\right]e_{k}=0\cdot e_{k}=0	\]
			for $j=1,\ldots,r$.  Thus, $z\mapsto F(z)x$ is constant and			
			$F^{(k)}(z)x=0$ whenever $x$ belongs to the linear span of 
			first $r$ columns of $V^{*}_{z_0}$, or equivalently,
			when $x$ satisfies $\|F(z_0)x\|=\|F(z_0)\|\cdot\|x\|$.
\end{proof}

\section{Minimum singular value principles.}\label{sectionMinSkPrinciples}

In the case of nonconstant scalar-valued functions, the MMP tells us that the 
\emph{minimum modulus} (of an analytic function on a region) can only be attained 
at a \emph{zero} of the function.  This conclusion is often called the \emph{minimum 
modulus principle} in complex analysis.  As a consequence of Theorem 
\ref{singValueMaxPrincipleThm}, we state and prove an analog of that minimum 
principle in the context of matrix-valued functions.

\begin{theorem}\label{MSVP}
			Let $\Omega$ be a region of $\C$ and let 
			$F:\Omega\to\M_n$ be a nonconstant analytic function.  Then no point 
			$z_0\in\Omega$ can be a minimum value for all of the functions $s_{k}(F(z))$, 
			$1\leq k\leq n$, unless $F(z_0)$ is not invertible.
\end{theorem}
\begin{proof}
			We prove that if there is a $z_0\in\Omega$ such that $F(z_0)$ is invertible 
			and the functions $z\mapsto s_k(F(z))$ attain their minimum at $z_0$ for 
			$k=1,\ldots,n$, then $F(z)$ must be a constant function.
			
			To begin, recall that the collection of invertible matrices is open.  This 
			implies $F(z)$ must be invertible for all $z$ sufficiently close to $z_0$. 
			So, $G(z)\Mydef F^{-1}(z)$ exists in some neighborhood $\Omega_0$ of $z_0$, 
			$\det F(z)$ is nonzero and analytic on $\Omega_0$, and the adjugate (or 
			transpose of the cofactor matrix) $\adj(F(z))$ of $F(z)$ is analytic on 
			$\Omega_0$.  Thus, $G(z)=F(z)^{-1}=\det^{-1}(F(z))\adj(F(z))$ is analytic 
			on $\Omega_0$ as well.
						
			By the singular value decomposition, at each $z\in\Omega_0$, the singular 
			values of $G(z)$ are the reciprocals of those of $F(z)$; more specifically,
			\[	s_{k}(G(z))=1/s_{n-k+1}(F(z))\;\text{ for }k=1,\ldots,n,
					\text{ and }z\in\Omega_0.	\]
			Therefore, the assumption of the theorem is equivalent to stating that 
			the functions $z\mapsto s_{k}(G(z))$ attain a maximum on $\Omega_0$ at 
			$z_0$. By Theorem \ref{singValueMaxPrincipleThm}, $G(z)$ and $F(z)$ must 
			be constant on $\Omega_0$.  Finally, applying the identity theorem (e.g.,
			\cite[Theorem 10.18]{R}) to each entry of $F(z)$ implies that $F(z)$ is 
			constant throughout $\Omega$, as desired.
\end{proof}

\begin{corollary}\label{detCondMSVP}
			Let $\Omega$ be a region of $\C$ and let 
			$F:\Omega\to\M_n$ be a nonconstant analytic function.  If every 
			function $s_{k}(F(z))$, $1\leq k\leq n$, attains a minimum value 
			at $z_0\in\Omega$, then $\det(F(z_0))=0$.
\end{corollary}

\begin{remark}
			Notice that $s_{n}(F(z_0))=0$ if and only if $z_0$ is a zero of 
			$\det F(z)$; indeed, with an SVD of $A\in\M_n$, we see that
			\begin{equation}\label{skProduct}
						\prod_{k=1}^{n} s_k(A)=|\det(A)|.
			\end{equation}
			Thus, Corollary \ref{detCondMSVP} states that if every function 
			$s_{k}(F(z))$, $1\leq k\leq n$, attains a minimum value at 
			$z_0\in\Omega$, then $s_n(F(z_0))=0$.
\end{remark}

To illustrate Theorem \ref{MSVP}, it suffices to take $F:\D\to\M_2$ as in 
\eqref{firstEx}; indeed, the functions $s_{1}(F(z))=1$ and $s_{2}(F(z))=|g(z)|$ 
attain their respective minimum values at any zero $z_0$ of $g$ and $F(z_0)$ is 
certainly not invertible.

In light of Theorems \ref{singValueMaxPrincipleThm} and \ref{MSVP}, one may 
ask whether the singular values of a matrix-valued analytic function could
attain minimum values at \emph{distinct} points.  The following result gives
an affirmative answer.

\begin{theorem}\label{toyExThm}
			If $F:\C\to\M_2$ denotes the function defined by
			\begin{equation}\label{toyEx}
						F(z)=\left[
						\begin{array}{cc}
									1	&	z\\
									0	&	z-1
						\end{array}\right],
			\end{equation}
			then $s_1(F(z))$ has a minimum at $z_1=0$ and $s_2(F(z))$ has a minimum 
			at $z_2=1$.
\end{theorem}
\begin{proof}
			The remark following the proof of Theorem \ref{MONPmaxVector} shows 
			that $s_1(F(z))$ has a minimum at $z_1=0$; indeed, $z\mapsto F(z)x_0$ 
			is constant when $x_0=[1,0]^{T}$.  On the other hand, if $z_2=1$, then
			\[	F(z_2)=\left[
					\begin{array}{cc}
								1	&	1\\
								0	&	0
					\end{array}\right]	\]
			satisfies $s_2(F(z_2))=0$ because $\det F(z_2)=0$.  In particular, 
			$s_2(F(z))$ has a minimum at $z_2=1$.
\end{proof}

Finally, it is worth mentioning that a singular value of a matrix function 
may attain its minimum value at specified locations.  For instance, when $g$ 
and $h$ are analytic, the function defined by
\begin{equation}
			K(z)=\left[
			\begin{array}{cc}
						g(z)	&	1\\
						0			&	h(z)
			\end{array}\right]
\end{equation}
satisfies $s_2(K(z))=0$ at every zero of $g$ and $h$.

\section{Return to the resolvent and matrix exponential.}\label{sectionExamples}

With the wisdom acquired about the norms and singular values of 
analytic matrix-valued functions, we now return to the resolvent and 
matrix exponential of a given matrix. To simplify our notation, let 
$R_{A}(z)$ denote the resolvent of $A\in\M_n$ at $z$, i.e.,  
\[	R_{A}(z)=(A-zI)^{-1}\;\text{ for }z\in\C\setminus\sigma(A).	\]
Also, set $L_{A}(z)=A-zI$.

Let $\Omega$ be a region of $\C\setminus\sigma(A)$. By Theorem 
\ref{singValueMaxPrincipleThm}, the singular values $s_k(L_A(z))$
and likewise $s_k(R_A(z))$ cannot all attain a maximum value on 
$\Omega$ as functions.  Recalling that 
\begin{equation}\label{skRaLa}
			s_k(R_A(z))=1/s_{n-k+1}(L_A(z))\;\text{ when }1\leq k\leq n,
\end{equation}
it follows that the functions $s_k(R_A(z))$ cannot all attain a maximum nor
a minimum on $\Omega$; in fact, this holds for $k=1$ and $k=n$, respectively, 
as shown below.\footnote{The first inequality in Theorem \ref{resolventThm} 
was observed by Daniluk \cite{D} for resolvents of operators on a complex 
Hilbert space.}

\begin{theorem}\label{resolventThm}
			If $A\in\M_{n}$ and $\Omega$ is any region of $\C\setminus\sigma(A)$, then
			\[	s_1(R_{A}(z))<\sup_{\zeta\in\Omega}s_1(R_{A}(\zeta))
					\;\text{ and }\;
					s_{n}(R_{A}(\zeta))>\inf_{\zeta\in\Omega}s_{n}(R_{A}(\zeta))
					\;\text{for all }z\in\Omega.	\]
			In particular, the functions $s_1(R_{A}(z))$ and $s_{n}(R_{A}(z))$
			are nonconstant on $\Omega$.
\end{theorem}
\begin{proof}
			To obtain a contradiction, assume instead there are points 
			$z_0,w_0\in\Omega$ such that 
			\[	s_1(R_{A}(z_0))\geq s_1(R_{A}(\zeta))\;\text{ or }\; 
					s_1(L_{A}(w_0))\geq s_1(L_{A}(\zeta))\;\text{ for all }\zeta\in\Omega	\] 
			(see \eqref{skRaLa}).  By Theorem \ref{MONPmaxVector}, 
			\begin{equation}\label{derivativeConditions}
						R^{\prime}_{A}(z_0)x_0=0\;\text{ or }\;L^{\prime}_{A}(w_0)y_0=0
			\end{equation}
			when $x_0$ and $y_0$ are maximizing vectors for $R_{A}(z_0)$ and
			$L_{A}(w_0)$, respectively.  On the other hand, as			
			\[	R_{A}(z)-R_{A}(z_0)=R_{A}(z)[(A-z_0I)-(A-zI)]R_{A}(z_0)	\]
			for any $z\in\Omega$, we have
			\[	R^{\prime}_{A}(z_0)
					=\lim_{z\to z_0}R_{A}(z)R_{A}(z_0)=R_{A}^2(z_0)	\]
			and clearly $L_{A}^{\prime}(w_0)=-I$.  However, these equations, 
			together with \eqref{derivativeConditions}, imply that
			\[	x_0=L_{A}^2(z_0)R_{A}^2(z_0)x_0=L_{A}^2(z_0)R^{\prime}_{A}(z_0)x_0=0	\]
			or $y_0=-L_{A}^{\prime}(w_0)y_0=0$, which are impossible because 
			$\|x_0\|=\|y_0\|=1$.
\end{proof}

We now turn to the matrix exponential.  Recall that given $T\in\M_n$,
the \textbf{matrix exponential} of $T$ is the $n\times n$ matrix defined by
\[	\exp(T)\Mydef\sum_{n=0}^{\infty}\frac{1}{n!}T^{n}.	\]
It is not difficult to verify that the series above converges for any $T\in\M_{n}$
(say, under the operator norm), and $\exp(T)$ is invertible in $\M_n$ with inverse 
$\exp(-T)$. 

For $A\in\M_n$, we see that the map $z\mapsto\exp(zA)$ is a well-defined matrix-valued
function, analytic on the entire complex plane $\C$, and
\[	\frac{d}{dz}[\exp(zA)]=A\exp(zA)=\exp(zA)A.	\]
Furthermore, a straightforward verification\footnote{Indeed, when $\re z>\|A\|$, the 
function $t\mapsto\exp(t(A-zI))$ has operator norm equal to $e^{-\re(zt)}\|\exp(At)\|$, 
which tends to zero as $t\to\infty$, and so the integral over $[0,\infty)$ of its 
derivative equals the identity matrix.} reveals that
\begin{equation}\label{resolventViaLaplace}
			(zI-A)^{-1}=\int_{0}^{\infty}e^{-zt}\exp(tA)\,dt\;\text{ when }\re z>\|A\|,
\end{equation}
while term-by-term integration of the power series representations for the exponential
and the resolvent gives
\begin{equation}\label{expViaResolvent}
			\exp(tA)=\frac{1}{2\pi i}\int_{\Gamma_{r}}e^{t\xi}(\xi I-A)^{-1}\,d\xi,
\end{equation}
where $\Gamma_{r}$ denotes any circle of radius $r>\|A\|$ centered at the origin.

In addition to the intimate relationship between the resolvent and the matrix exponential
(as described in \eqref{resolventViaLaplace} and \eqref{expViaResolvent}), intuition from 
the case of scalar-valued functions may suggest that, in analogy to Theorem 
\ref{resolventThm}, the functions $s_1(\exp(zA))$ and $s_n(\exp(zA))$ should 
not attain their maximum and minimum values, respectively,\footnote{After all, 
for fixed $a\in\C$, $|e^{za}|$ cannot attain its maximum nor minimum values
over any region $\Omega$.} over any region $\Omega$ of $\C$.  This is in fact
\emph{false}.  Notice that
\[	\exp(zA)=\left[
		\begin{array}{cc}
					1	&	0\\
					0	&	e^{z}
		\end{array}\right]
		\;\text{ with }\;
		A=\left[
		\begin{array}{cc}
					0	&	0\\
					0	&	1
		\end{array}\right]	\]
provides a counterexample; indeed, computation reveals that
\[	s_{1}(\exp(zA))=\max\{1,e^{\re z}\}\;\text{ and }\;
		s_{2}(\exp(zA))=\min\{1,e^{\re z}\}.	\]
Thus, $s_{1}(\exp(zA))$ and $s_2(\exp(zA))$ are constant when $\re z<0$ 
and $\re z>0$, respectively.

Finally, we would like to propose a question for further investigation.
Given an analytic function $F:\Omega\to\M_n$ such that $\|F(z)\|$ attains 
its maximum in $\Omega$, Theorem \ref{factoringFzThm} not only describes 
the structure of $F$, it also implies that $\|F(z)\|=\|F(z_0)\|$ for all $z\in
\Omega$.  So, in a sense, it is rare for $\|F(z)\|$ to attain its maximum.  
Instead, what may be less rare is for $\|F(z)\|$ to attain a \emph{minimum} 
value (see Theorem \ref{toyExThm}).  

In fact, the remark made after the proof of Theorem \ref{MONPmaxVector}
already gives a sufficient condition for $\|F(z)\|$ to have a minimum at 
$z_0$, namely when $z\mapsto F(z)x_0$ is constant for some maximizing 
vector for $F(z_0)$.  Furthermore, by completely analogous reasoning, 
a sufficient condition for $s_{n}(F(z))$ to attain a maximum at $z_0$
is that $z\mapsto F(z)x_0$ is constant for some minimizing\footnote{A 
vector $x_0$ is said to be a \textbf{minimizing vector} for $A\in\M_n$ 
if $\|Ax_0\|=\min\{\|Ax\|:\|x\|=1\}$.} vector for $F(z_0)$.  For example, 
in light of this, it may be verified that $s_1(F(z))$ has a minimum at 
$z=0$ and $s_{2}(F(z))$ has a maximum at $z=0$ when $F(z)$ is the function 
in \eqref{toyEx}. This leads one to wonder what necessary and sufficient 
conditions permit $s_{1}(F(z))$ to attain a minimum and $s_{n}(F(z))$ to 
attain a maximum over a region $\Omega$?  Is it more attainable to consider 
the special case $F(z)=(A-zI)^{-1}$?  How about when $F(z)=\exp(zA)$?\medskip

\textbf{Acknowledgments.}
I wish to thank Cara D. Brooks whose valuable comments helped improve 
the exposition of the paper tremendously.  Also, I am grateful to Nicholas 
Seguin whose sustained interest in this project helped me see it to
completion.  Finally, I thank the referees and editor for their helpful
suggestions.

\end{document}